\documentclass[a4paper]{scrartcl}
\usepackage[]{amsmath,amssymb} \usepackage{amsthm} \usepackage{enumitem}
\usepackage[utf8]{inputenc} \usepackage[T1]{fontenc} \usepackage{lmodern}
\usepackage{fourier} \usepackage{graphicx}
  \def\dif{\,\D} 
 \usepackage{pgfplots}
\pgfplotsset{compat=1.15}
\usepackage{tikz}
\usetikzlibrary{arrows}
\def\R{{\mathbf R}}  \def\N{{\mathbf N}} 

\def\car#1{{\mathbf 1}} 
\newcommand{\esp}[1]{{\mathbf E}\left[#1\right]} 

\def\P{\mathbf P}  
\def\/{\,|\,} 
\newcommand{\var}{\operatorname{Var}}
\newcommand{\cov}{\operatorname{Cov}}
 \def\NN{\mathfrak N} 
 \def\indic{\mathbf 1}
\def\/{\,  |\,}
 \def\car{\mathbf 1} 
\def\history{\mathcal F}  
\def\l{\mathfrak l}  \def\L{\operatorname{L}}

 \def\P{\mathbf{P}} \def\R{\mathbf{R}} 
\def\N{\mathbf{N}} \def\1{\mathbf{1}}

\def\cov{\operatorname{cov}}  \def\L{\operatorname{L}}
 
 \def\dif{\operatorname{d}\!}

\def\dist{\operatorname{dist}}

\definecolor{fonce}{HTML}{415458} \definecolor{cadet}{HTML}{5D737E}
\definecolor{vertdeau}{HTML}{7E795D} \definecolor{ocre}{HTML}{58414F}
\definecolor{creme}{HTML}{7E5D72}

\RequirePackage[framemethod=default]{mdframed} \newmdenv[skipabove=7pt,
skipbelow=7pt, rightline=false, leftline=true, topline=false, bottomline=false,
backgroundcolor=ocre!10, linecolor=ocre, innerleftmargin=5pt,
innerrightmargin=20pt, innertopmargin=5pt, innerbottommargin=5pt,
leftmargin=0cm, rightmargin=0cm, linewidth=4pt]{eBox}

\makeatletter
\newtheoremstyle{ocrenumbox}
{0pt}
{0pt}
{\itshape}
{}
{\small\bfseries\sffamily\color{cadet}}
{\;}
{0.25em}
{\small\sffamily\color{ocre}\thmname{#1}\nobreakspace\thmnumber{\@ifnotempty{#1}{}\@upn{#2}}
  \thmnote{\nobreakspace\the\thm@notefont\sffamily\bfseries\color{black}---\nobreakspace#3.}} 

\newtheoremstyle{blacknumex}
{5pt}
{5pt}
{\normalfont}
{} 
{\small\bfseries\sffamily}
{\;}
{0.25em}
{\small\sffamily{\tiny\ensuremath{\blacksquare}}\nobreakspace\thmname{#1}\nobreakspace\thmnumber{\@ifnotempty{#1}{}\@upn{#2}}
  \thmnote{\nobreakspace\the\thm@notefont\sffamily\bfseries---\nobreakspace#3.}}

\newtheoremstyle{blacknumbox} 
{0pt}
{0pt}
{\small} {}
{\small\bfseries\sffamily}
{~}
{0.25em}
{\small\sffamily\thmname{#1}\nobreakspace\thmnumber{\@ifnotempty{#1}{}\@upn{#2}}
  \thmnote{\nobreakspace\the\thm@notefont\sffamily\bfseries---\nobreakspace#3.}}

\newtheoremstyle{ocrenum}
{5pt}
{5pt}
{\normalfont}
{}
{\small\bfseries\sffamily\color{ocre}}
{\;}
{0.25em}
{\small\sffamily\color{ocre}\thmname{#1}\nobreakspace\thmnumber{\@ifnotempty{#1}{}\@upn{#2}}
  \thmnote{\nobreakspace\the\thm@notefont\sffamily\bfseries\color{black}---\nobreakspace#3.}} 
\makeatother

\newmdenv[skipabove=7pt, skipbelow=7pt, rightline=false, leftline=true,
topline=false, backgroundcolor=ocre!10, bottomline=false, linecolor=ocre,
innerleftmargin=5pt, innerrightmargin=20pt, innertopmargin=5pt,
innerbottommargin=5pt, leftmargin=0cm, rightmargin=0cm, linewidth=4pt]{dBox}

\newmdenv[skipabove=7pt, skipbelow=7pt, skipbelow=7pt, backgroundcolor=fonce!10,
rightline=false, leftline=true, bottomline=false, topline=false, linewidth=4pt,
linecolor=fonce, innerleftmargin=5pt, innerrightmargin=20pt, innertopmargin=5pt,
leftmargin=0cm, rightmargin=0cm, innerbottommargin=5pt]{tBox}

\theoremstyle{ocrenumbox} \newtheorem{theoremeT}{Theorem}
\newtheorem{lemmaT}[theoremeT]{Lemma}
\newtheorem{corollaryT}[theoremeT]{Corollary}
\theoremstyle{blacknumex}
\newtheorem{exampleT}{Example} \newtheorem{remarkT}{Remark}

\theoremstyle{blacknumbox}
\newtheorem{definitionT}{Definition}
\newenvironment{theorem}{\begin{tBox}\begin{theoremeT}}{\end{theoremeT}\end{tBox}}
\newenvironment{lemma}{\begin{tBox}\begin{lemmaT}}{\end{lemmaT}\end{tBox}}
\newenvironment{definition}{\begin{dBox}\begin{definitionT}}{\end{definitionT}\end{dBox}}

\numberwithin{theoremeT}{section} \numberwithin{definitionT}{section}

\def\<{\left\langle} \def\>{\right\rangle}
\def\Lip{\operatorname{Lip}}
\def\KR{\text{\textsc{kr}}}

\def\card{\operatorname{card}}

\def\<{\left\langle} \def\>{\right\rangle}
\def\Lip{\operatorname{Lip}}
\def\KR{\text{\textsc{kr}}}
\def\LP{\text{\textsc{lp}}}
\def\R{{\mathbf R}}

\def\N{{\mathbf N}}

\def\car#1{{\mathbf 1}}

\def\P{\mathbf P}

\def\/{\,|\,} 

\def\dist{\operatorname{dist}}
\def\dif{\text{ d}}

\def\L{\operatorname{L}}
\def\history{\mathcal{F}}
\def\indic{\mathbf{1}}

\title{Stein's functional method for weakly dependent random variables} \author{L.
  Coutin \and S. Darou Kebe \and L. Decreusefond} \date{2026}
\begin{document}
\maketitle{}
\begin{abstract}
  This article extends Stein's functional method to establish explicit
  rates of convergence in the Wasserstein-1 distance for Donsker's invariance
  principle under weak dependence. Our approach first reduces the problem of
  bounding distances between distributions on a function space to estimating
  distances between finite-dimensional marginals. We then employ the block technique, which is standard in Stein's methodology, to handle dependence. While our analysis focuses on
  $\phi$-mixing sequences, the method extends to other forms of weak dependence
  admitting suitable covariance inequalities.
\end{abstract}


 \noindent{Math Subject Classification: 60F17}

\section{Motivations}
\label{sec:motivations}
\def\law{\mathcal{L}}

Donsker's theorem is a cornerstone of functional central
limit theory, establishing the convergence of random walks to Brownian motion
under independence assumptions. However, modern applications in time series,
spatial data, and stochastic processes in physics and finance necessitate
relaxing independence to weak dependence structures, such as mixing conditions
(e.g., $\alpha$-, $\beta$-, $\rho$-mixing). While sharp rates of convergence for
independent variables are well understood, analogous quantitative bounds for
dependent variables---especially in functional spaces---remain scarce. This work
is motivated by the need to extend the results of \cite{Coutin2013}, which
assessed the rate of convergence in Donsker's theorem for independent
random variables, to weakly dependent variables.

There are more than seventy notions of distance between probability measures
(see \cite{Rachev1998}), but many of them are variations on the L\'evy--Prokhorov
distance or on the Wasserstein-1 (also known as Kantorovitch--Rubinstein)
distance. For $\mu$ and $\nu$ two probability measures on a metric space $(E,d)$
with Borel $\sigma$-field $\mathcal{A}$, the L\'evy--Prokhorov distance is
defined as
\begin{multline*}
  \dist_{\LP}(\mu,\, \nu)=\max\Bigl( \inf\bigl\{\epsilon>0\, :\, \mu(A)\le
  \nu(A^{\epsilon})+\epsilon \ \text{ for all closed } A\subset E\bigr\},\\
  \inf\bigl\{\epsilon>0\, :\, \nu(A)\le \mu(A^{\epsilon})+\epsilon\ \text{ for all closed
  }A\subset E \bigr\}\Bigr)
\end{multline*}
where $A^{\epsilon}=\{x\in E\, :\, d(x,A)\le \epsilon\}$.
The Kantorovitch--Rubinstein (or Wasserstein-$1$) distance is defined as
\begin{equation*}
  \dist_{\KR(E)}(\mu,\, \nu)=\sup_{f\in \operatorname{Lip}_{1}(E,d)}\left( \int_{E}f\dif \mu-\int_{E}f\dif \nu \right)
\end{equation*}
where $\operatorname{Lip}_{1}(E,d)$ is the set of $1$-Lipschitz functions:
\begin{equation*}
  \operatorname{Lip}_{1}(E,d)=\Bigl\{f\,:\, E\to \R, |f(x)-f(y)|\le d(x,y),\, \forall x,y\in E\Bigr\}.
\end{equation*}
The L\'evy--Prokhorov distance metrizes convergence in distribution. The
Kantorovitch--Rubinstein distance is more stringent, as
$\dist_{\KR(E)}(\mu_{n},\,\mu)$ tends to $0$ if and only if $\mu_{n}$ converges
narrowly to $\mu$ and, for any $x_{0}\in E$,
\begin{equation*}
  \int_{E}d(x_{0},x)\dif \mu_{n}(x)\xrightarrow{n\to\infty} \int_{E}d(x_{0},x)\dif \mu(x).
\end{equation*}
For real-valued random variables, the $\L^{1}$ distance between distribution
functions,
\begin{equation*}
  \int_{\R}\left|\mu\bigl((-\infty,x]\bigr)-\nu\bigl((-\infty,x]\bigr)\right|\dif x,
\end{equation*}
which in fact coincides with $\dist_{\KR(\R)}(\mu,\nu)$, is often invoked. Early results (e.g., \cite{agnew1957estimates,esseen1958mean})
established rates of convergence for the classical central limit theorem in this
metric. In the functional setting (see \cite{prokhorov_1956,Komlos1975}),
Donsker's theorem (or equivalently, the invariance principle) states that for
independent and identically distributed random variables $(X_{k},\ k\ge 1)$ with
$\esp{X_{k}} = 0$, $\var(X_{k}) = 1$, and $\esp{|X_k|^p} < \infty$ for some
$p = 2+2\delta'>2$, we have, in the L\'evy--Prokhorov distance,
\begin{equation}
  \label{eq_main_bis:11}\dist_{\LP}\Bigl(\law(X^{n}),\ \mu\Bigr)=O(n^{-\delta'/(3+2\delta')}),
\end{equation}
where $\mu$ is the Wiener measure on the set of continuous functions from
$[0,1]$ to $\R$ and $X^{n}$ is the affine interpolation of the random walk
\begin{equation*}
  X^{n}(t)=\sqrt{n}\ \sum_{k=1}^n X_k \, r_k^n(t)
\end{equation*}
where
\begin{equation*}
  r_k^n(t)=\int_0^t \indic_{[k/n,\ (k+1)/n)}(s)\dif s.
\end{equation*}
In \cite{coutin2020stein}, it is shown that under the hypothesis of the
existence of third moments, we have for $(\eta,p)$ such that $1/p<\eta<1/2$,
\begin{equation*}
  \dist_{\KR(W_{\eta,p})}\left(\law(X^{n}),\ \mu\right)=O\left(n^{-1/6+\eta/3}\log n\right)
\end{equation*}
where $W_{\eta,p}$ is a fractional Sobolev space defined below, whose
particularly interesting feature is to be continuously embedded in the set of
$(\eta-1/p)$-H\"older continuous
functions. 
If we choose $p=3$ and $\eta$ close to $1/3$, we obtain a rate of convergence slightly
worse than $n^{-1/18}\log n$, instead of the rate $n^{-1/8}$ obtained
in~\cite{borovkov_1973,Komlos1975}. Whether this gap arises from the
Kantorovitch--Rubinstein distance being stronger than the L\'evy--Prokhorov metric,
or from a methodological difference, remains an open question. A natural
extension of these results is to relax the independence hypothesis. In
\cite{borovkov_1973,ibragimov_linnik_1971}, the authors developed estimates for
the rate of convergence in the L\'evy--Prokhorov distance for stationary or
non-stationary $\phi$-mixing sequences. The papers \cite{yurinskii_1976, zolotarev_1977,
  berkes_philipp_1979} contributed to the estimation of distances between
distributions and the approximation of sums of dependent random vectors. In
\cite{Haeusler1984}, it is proved that \eqref{eq_main_bis:11} still holds when
$(X_{k},\, k\ge 1)$ is a real-valued functional of a Doeblin process, hence a
geometrically $\phi$-mixing sequence. Most of these papers
are based on the following elementary inequality (see \cite{Sawyer1972}):
\begin{equation*}
  \dist_{\LP}\left(\law(X),\law(Y)\right)\le \epsilon +\P(d(X',Y')\ge \epsilon)
\end{equation*}
for any $\epsilon>0$, where $(X',Y')$ is a coupling of $X$ and $Y$. A
coupling sharp enough for this purpose is then obtained through the Skorohod
embedding theorem.
Stein's method is a completely different approach and is restricted to
estimating distances that can be expressed in the form
\begin{equation*}
  \dist_{\mathcal{H}}(\mu,\, \nu)=\sup_{f\in \mathcal{H}}\left(\int_{E}f\dif \mu-\int_{E}f\dif \nu\right)
\end{equation*}
where $\mathcal{H}$ is a set of test functions (bounded and Lipschitz for the
Fortet--Mourier distance, bounded for the total variation distance, indicator
functions for the Kolmogorov distance, etc.).
Since its early development in the 1970s, Stein's method has been known to
accommodate dependence between random variables well; see
\cite{ross2011fundamentals} for a survey. For the classical CLT, we refer to
\cite{baldi1989normal,sunklodas2007normal,takahata1983central} for rates of
convergence under different hypotheses of dependence and conditions of
integrability. The framework of \cite{sunklodas2007normal} and others
shows that the polynomial decay of the mixing coefficients governs the rate of
convergence in Wasserstein distance.

In functional settings, the literature on Stein's method is much sparser. Since
the pioneering work of Barbour~\cite{Barbour1990}, one may cite
\cite{Coutin2013,coutin2020stein,Bourguin2020,Shih201109}. Two main strategies
are used to assess the rate of convergence towards a process such as Brownian
motion: using the characterization with the generator of the limiting Markov
process, as in \cite{Coutin2013,Shih201109,Bourguin2020,zbMATH07873656}, or
using a method similar to the classical tightness plus convergence of
finite-dimensional distributions, which requires estimates of the modulus of
continuity of the target process \cite{Coutin2013,Coutin2025b}. Here, we choose
the latter, as it is much easier to handle. We now explain this strategy of proof. We have to compare the law
\begin{equation*}
  X^n(t)=\sum_{k=0}^{2^{n}-1} X_k \, h_{k}^n(t)
\end{equation*}
where $h_{k}^{n}=2^{n/2}r_{k}^{n}$ to the law of $\sigma B$, where $\sigma>0$ and
$B$ is an ordinary Brownian motion. We first consider the affine interpolation
of both processes, $\pi_{m}X^{n}$ and $\pi_{m}(\sigma B)$. They can be
written (see below for the details) as
\begin{equation*}
  \pi_{m}X^{n}=\sum_{i=0}^{2^{m}-1}U_{m,i}h_{i}^{m}\text{ and } \pi_{m}(\sigma B)=\sum_{i=0}^{2^{m}-1}B_{m,i}h_{i}^{m}
\end{equation*}
where the $U_{m,i}$ and $B_{m,i}$ are real-valued random variables. One of the
seldom used properties of the Kantorovitch--Rubinstein metric is that it admits
the alternative representation:
\begin{equation*}
  \dist_{\KR(E)}(\mu,\nu)=\inf_{X\sim\mu,Y\sim \nu}\esp{d(X,Y)}.
\end{equation*}
If $\Bigl((\hat{U}_{m,i},\hat{B}_{m,i}),\, 0\le i\le 2^{m}-1\Bigr)$ is a coupling of
$(U_{m,i},\, 0\le i \le 2^{m}-1)$ and $(B_{m,i},\, 0\le i \le 2^{m}-1)$ for the
$\ell^{1}$ distance on $\R^{2^{m}}$, we have
\begin{multline*}
  \dist_{\KR(W_{\eta,p})}\Bigl(\law( \pi_{m}X^{n}),\ \law(\pi_{m}(\sigma B))\Bigr)\\
  \begin{aligned}
     & =\sup_{F\in \Lip_{1}(W_{\eta,p})}\left(\esp{F(\pi_{m}X^{n})}-\esp{F(\pi_{m}(\sigma B))}\right)                                                      \\
     & =\sup_{F\in \Lip_{1}(W_{\eta,p})}\left(\esp{F(\sum_{i=0}^{2^{m}-1}\hat U_{m,i}h_{i}^{m})}-\esp{F(\sum_{i=0}^{2^{m}-1}\hat B_{m,i}h_{i}^{m})}\right) \\
     & \le \esp{ \left\|\sum_{i=0}^{2^{m}-1} (\hat U_{m,i}-\hat B_{m,i})h_{i}^{m}\right\|_{W_{\eta,p}} }                                                   \\
     & \le \sum_{i=0}^{2^{m}-1} \|h_{i}^{m}\|_{W_{\eta,p}}\ \dist_{\KR(\R)}\bigl(\law(U_{m,i}),\, \law(B_{m,i})\bigr).
  \end{aligned}
\end{multline*}
Provided that we can estimate the distance between the original processes and
their affine interpolations of index $m$, we are led to a one-dimensional problem
of the kind Stein's method is designed to handle. Indeed, up to a normalization,
each $U_{m,i}$ is the sum of the $X_{j}$ over the $i$-th block of the dyadic
subdivision of mesh $2^{-m}$, hence is of the form
\begin{equation*}
  \frac{1}{\sqrt{q}}\sum_{j=1}^{q} X_{j},
\end{equation*}
where $q$ denotes the common size of the blocks and the $X_{j}$ are our initial
random variables. The random variables
$B_{m,i}$ are simply centered Gaussian random variables of variance
$\sigma^{2}$. It is well known that
\begin{multline*}
  \dist_{\KR(\R)}\left(\law\Bigl( \frac{1}{\sqrt{q}}\sum_{j=1}^{q} X_{j}\Bigr),\, \law(\mathcal{N}(0,\sigma^{2}))\right)
  \\  =\sup_{f\in \mathcal{H}}\left(\esp{\frac{1}{\sqrt{q}}\sum_{j=1}^{q} X_{j}f\Bigl(\frac{1}{\sqrt{q}}\sum_{k=1}^{q} X_{k}\Bigr)}-\sigma^{2}\esp{f'\Bigl(\frac{1}{\sqrt{q}}\sum_{j=1}^{q} X_{j}\Bigr)}\right)\\
  =\sup_{f\in \mathcal{H}}\left(  A_{1}(f)-A_{2}(f) \right),
\end{multline*}
where $\mathcal{H}$ is the set of $\mathcal{C}^{2}_{b}$ functions with
$\|f^{\prime \prime}\|_{\infty}\le (2/\pi)^{1/2}$. The usual way to proceed is
to transform $A_{1}(f)$ so that it looks like $A_{2}(f)$. When the $X_{i}$ are
independent, we generally use the leave-one-out method, which is
based on the remark that $X_{i}$ and $\sum_{j=1}^{q} X_{j}-X_{i}$ are
independent, so that we have
\begin{equation*}
  A_{1}(f)= \frac{1}{\sqrt{q}}\sum_{j=1}^{q} \esp{ X_{j}\left(f\Bigl(\frac{1}{\sqrt{q}}\sum_{k=1}^{q} X_{k}\Bigr)-f\Bigl(\frac{1}{\sqrt{q}}\sum_{k\neq j} X_{k}\Bigr)\right)}.
\end{equation*}
The remainder of the proof proceeds via Taylor
expansions where necessary. In the presence of dependence, this scheme no longer
works and we resort to the so-called block method (see
\cite{takahata1983central}). Instead of removing only one random variable, we
remove a whole block, of a size chosen so as to make the blocks almost
independent, see Theorem~\ref{thm_mixing_wiener:decomposition}. This is where
the type of weak dependence enters: the rate of convergence varies with the
strength of the correlations. We work here with $\phi$-mixing, but the method
carries over to any other kind of mixing condition. The difficulty lies in the
optimization problem we have to solve to obtain a good rate of convergence. We
actually have to strike a balance between the step of the affine interpolation
and the size of the blocks. If the step of the interpolation is too small, there
are not enough random variables for the CLT to be applicable; if it is too
large, there may be a huge difference between the original process and its
interpolation. If the block size is too small, the correlations are too strong
to be controlled; if it is too large, we are back to the original problem. One
last word before turning to the core of the paper: we work here with the dyadic
affine interpolation in order to suppress some side effects which appear when
using the usual interpolation of step $1/m$. This can be addressed as in \cite{coutin2020stein};
however, for simplicity, we choose to avoid this additional technicality.

The paper is organized as follows. In Section~\ref{sec:preliminaries}, we
introduce the concept of $\phi$-mixing, along with the maximal and covariance
inequalities that underpin our analysis. Section~\ref{sec:convergence-rate} is
dedicated to establishing convergence results. Therein, we first assess the
distance between the original process and its affine interpolation. Leveraging
the properties of the Kantorovitch--Rubinstein distance, we then reduce the
problem to quantifying the convergence rate of certain finite-dimensional
distributions towards a Gaussian vector. Finally, the optimal rate is derived by
optimizing the step size of the affine interpolation.

\section{Preliminaries}
\def\Hol{\operatorname{Hol}}
\label{sec:preliminaries}
In the study of functional convergence, it is essential to first specify the
functional spaces within which the rate of convergence is to be evaluated. Here,
we consider the Slobodeckij spaces, which,
for $p\in [1,\infty)$ and $\eta\in (0,1)$, are defined as
\begin{equation*}
  W_{\eta,p}=\Bigl\{f\, :\, [0,1]\to \R,\ \iint_{[0,1]^{2}}\frac{|f(t)-f(s)|^{p}}{|t-s|^{1+\eta p}}\dif s \dif t<\infty\Bigr\}.
\end{equation*}
When equipped with the norm
\begin{equation*}
  \|f\|_{W_{\eta,p}}=\|f\|_{\L^{p}([0,1])}+\left(\iint_{[0,1]^{2}}\frac{|f(t)-f(s)|^{p}}{|t-s|^{1+\eta p}}\dif s \dif t\right)^{1/p},
\end{equation*}
$W_{\eta,p}$ becomes a separable Banach space. Furthermore, it satisfies the
following Sobolev embedding:
\begin{equation*}
  \eta-\frac{1}{p}>0 \Longrightarrow W_{\eta,p}\subset \Hol(\eta-1/p)
\end{equation*}
where $\Hol(\lambda)$ is the space of $\lambda$-H\"older continuous functions on
$[0,1]$. It is well known that Brownian motion has sample paths which belong
almost surely to $W_{\eta,p}$ for any $\eta<1/2$ and $p>1/\eta$.

All the
random variables considered are defined on a probability space
$(\Omega, \mathcal{A}, \P )$. We denote by $\L^{p}$ the set of $p$-integrable
random variables on~$(\Omega, \mathcal{A}, \P )$. For a sub-$\sigma$-field
$\mathcal{G}\subset \mathcal{A}$, we denote by $\L^{p}(\mathcal{G})$ the
set of random variables which belong to $\L^{p}$ and are $\mathcal{G}$-measurable.
For a finite set $V$, $|V|$ denotes its
cardinality. We denote by $\mathcal{U}$ the set of finite subsets of $\N$.
For $V_{1}, V_{2}\in \mathcal{U}$, we set
\begin{equation*}
  d(V_1, V_2)= \min\{ |a-b|\, :\, a\in V_1,\ b\in V_2 \}.
\end{equation*}
For $V\in\mathcal{U}$, we set $\history(V)= \sigma\{X_a,\, a\in V\}$ if
$V\neq \emptyset$, and $\history(\emptyset)=\{\emptyset,\Omega\}$.
For $V_{1},V_{2}$ in $\mathcal{U}$,
let
\begin{equation*}
  \varphi(V_{1},V_{2})=\sup_{A\in \history(V_{1}),B\in \history(V_{2}),\P(A)>0}\left|\P(B\/A)-\P(B)\right|,
\end{equation*}
and define the $\phi$-mixing coefficients by
\begin{equation}
  \label{eq_03-main-results:1}\phi(n)=\sup_{k\ge 1}\ \sup_{l\ge 0}\ \varphi\Bigl(\{0,\dots,k\},\ \{k+n,\dots,k+n+l\}\Bigr).
\end{equation}
\begin{definition}
  A sequence $X=(X_{n},\, n\ge 0)$ is said to be $\phi$-mixing if the sequence
  $\phi$ defined in~\eqref{eq_03-main-results:1} tends to $0$ as $n$ goes to
  infinity.
\end{definition}
\begin{definition}
  We denote by $\mathcal{W}_{\delta}$ the set of sequences
  $X=(X_{n},\, n\ge 0)$ which are centered, weakly stationary in $\L^{2}$ and
  bounded in $\L^{2+\delta}$ for some $\delta>0$, i.e.,
  \begin{equation}\label{C11_ineq}
    \sup_{k\in \N} \|X_{k}\|_{\L ^{2+\delta}}<\infty.
  \end{equation}
  We denote by $\mathcal{W}_{\delta,\phi}$ the set of sequences in
  $\mathcal{W}_{\delta}$ which are $\phi$-mixing with
  \begin{equation*}
    \sum_{n\ge 0} \phi(n)^{1/2}<\infty.
  \end{equation*}
\end{definition}
The following property of $\phi$-mixing sequences is established
in~\cite{nahapetian1980central}.
\begin{theorem}
  \label{thm:nahapetian}
  For $X\in \mathcal{W}_{\delta,\phi}$, there exists a non-decreasing function
  $M\, :\, \N\to [1,+\infty)$ such that, for any $V_1, V_2 \in \mathcal{U}$, we
  have
  \begin{equation}\label{cond-mixing}
    \sup_{A\in\history(V_1)}\bigl|\P(A\/\history(V_2)) -\P (A)\bigr|\leq M\bigl(|V_1|\bigr)\ \phi\bigl(d(V_1,V_2)\bigr).
  \end{equation}
\end{theorem}
Given $\phi$ and $\delta$ as above, in what follows, $X$ is a fixed element of
$\mathcal{W}_{\delta,\phi}$.
We then require two technical results, which are the two cornerstones of
our subsequent analysis. Specifically, whenever the covariance inequality and
the maximal inequality are satisfied, we are in a position to derive a
convergence rate for Donsker's invariance principle. The first theorem can be
found for instance in~\cite{takahata1983central}.
\begin{theorem}[Covariance inequality]
  \label{lem-hypo} Let $X\in \mathcal{W}_{\delta,\phi}$. Let $r\in (1,+\infty]$ and let $s$ denote its conjugate
  exponent. Let
  $V_1, V_2 \in \mathcal{U}$. Assume that $Y\in \L^{r}\bigl(\history(V_{1})\bigr)$ and $Z\in \L^{s}\bigl(\history(V_{2})\bigr)$.
  Then, we have
  \begin{equation}\label{lem-cov}
    |\cov(Y, Z)|\leq 2 \ M^{\frac{1}{s}}(|V_1|)\ \phi^{\frac{1}{s}}\bigl(d(V_1,V_2)\bigr)\, \|Y\|_{\L^{r}}\, \|Z\|_{\L^{s}}.
  \end{equation}
\end{theorem}
According to~\cite{wang2009}, we
have the following maximal inequality for  partial sums of a $\phi$-mixing sequence:
\begin{theorem}[Maximal inequality]
  \label{thm:mixing_wiener:maximal_inequality}
  Let $X\in \mathcal{W}_{\delta,\phi}$. For $k\ge 0$ and $m\ge 1$, let
  \begin{equation*}
    S_k(i)=\sum_{j=k}^{k+i-1} X_j.
  \end{equation*}
  For any $p\in [2,2+\delta]$ with $\delta>0$, we have
  \begin{equation*}
    \esp{\max_{1\le i\le m} |S_k(i)|^p}\lesssim\,  m^{p/2}\sup_{k\le i< k+m}\|X_i\|_{\L^{p}}^p.
  \end{equation*}
\end{theorem}
Here and hereafter, the expression $a\lesssim b$ means that there exists a
constant $K>0$, independent of $k$, $m$ and $n$, such that $a\le Kb$.

To quantify the rate of convergence, we need an assumption on the speed at which
$\phi(n)$ goes to zero. In what follows, this assumption is always in force.
\begin{definition}
  Let
  \begin{equation*}
    \rho= \frac{1+\delta}{2+\delta}\in \Bigl(\frac{1}{2},1\Bigr).
  \end{equation*}
  For $\theta>0$ such that $\rho\theta>2$, we denote by
  $\mathcal{W}_{\delta,\phi}^{\theta}$ the set of random sequences
  $X\in \mathcal{W}_{\delta,\phi}$ which satisfy
  \begin{equation}
    \label{eq_03-main-results:2}\phi(n)\le n^{-\theta},\qquad n\ge 1.
  \end{equation}
\end{definition}
Note that $\rho\theta>2$ implies $\theta>2$. Since the sequence $X$ is weakly stationary in $\L^{2}$, we can define
\begin{align*}
  \Gamma_{j}=\esp{X_{n}X_{n+j}}=\esp{X_{0}X_{j}}, \text{ for any
                                                  } n\in \N,\ j\in \N.
\end{align*}
\begin{lemma}
  Let $X\in \mathcal{W}_{\delta,\phi}^{\theta}$. Then, $\Gamma$ belongs to
  $\ell^{1/2}(\N)$ and
  \begin{equation*}
    \sigma^{2}=\Gamma_{0}+2\sum_{j\ge 1}\Gamma_{j}\in [0,+\infty).
  \end{equation*}
  Throughout the paper, we assume that $\sigma^{2}>0$.
\end{lemma}
\begin{proof}
  According to~\eqref{lem-cov} and~\eqref{C11_ineq}, we have
  \begin{align*}
    \left|\Gamma_{j}\right| & \lesssim \left(\phi(j)\right)^{\frac{1 + \delta}{2 + \delta}} \ \|X_{0}\|_{\L^{(2+\delta)/(1+\delta)}}\, \|X_{j}\|_{\L^{2+\delta}} \\
                            & \lesssim  j^{-\rho\theta}.
  \end{align*}
  Since $\rho\theta>2$, $\Gamma$ belongs to $\ell^{1/2}(\N)$; moreover, it is
  well known that the sum of the covariances of a weakly stationary sequence is
  non-negative.
\end{proof}

\section{Convergence rate}
\label{sec:convergence-rate}
Since we deal with dyadic partitions, we slightly change the notation
compared to Section~\ref{sec:motivations}. The random walk associated with
the sequence $(X_k,\, k\ge 0)$ is defined by
\begin{equation*}
  X^n(t)=2^{n/2}\ \sum_{k=0}^{2^{n}-1} X_k \, r_k^n(t) = \sum_{k=0}^{2^{n}-1} X_k \, h_{k}^n(t)
\end{equation*}
where
\begin{equation*}
  \label{eq_donsker:11}  r_k^n(t)=\int_0^t \indic_{[k2^{-n},\ (k+1)2^{-n})}(s)\dif s \text{ and }
  h_{k}^n(t)=2^{n/2}\ r_k^n(t).
\end{equation*}
We recall from~\cite{coutin2020stein} that
\begin{equation*}
  \|h_{k}^n\|_{W_{\eta,p}}\lesssim 2^{-n(\frac{1}{2}-\eta)}.
\end{equation*}
\subsection{Affine interpolation}
\label{sec:Affine interpolation}

For $m\ge 1$, we denote by $\pi_{m}$ the orthogonal projection onto
$\operatorname{span}(h_{l}^{m},\, l=0,\dots,2^{m}-1)$ with respect to the
scalar product $(f,g)\longmapsto \int_{0}^{1}f'(s)g'(s)\dif s$. It
corresponds to the affine interpolation on the regular subdivision of
step~$2^{-m}$. For $m<n$,
we have
\begin{equation*}
  \pi_{m}X^{n}=  \sum_{j=0}^{2^{m}-1}\left(\frac{1}{2^{(n-m)/2}}\sum_{l=j2^{n-m}}^{(j+1)2^{n-m}-1}X_{l} \right)h_{j}^{m}.
\end{equation*}
Although the random variables $X_{i}$ are not independent,
Theorem~\ref{thm:mixing_wiener:maximal_inequality} yields the same upper bound
on the modulus of continuity of  $X^{n}$ as in \cite{Coutin2020}:
\begin{theorem}
  \label{thm:majo_var}
  Let $X\in \mathcal{W}_{\delta,\phi}$. For any  $p\in [2,2+\delta]$ with $\delta>0$,
  we have:
  \begin{equation*}
    \sup_{n\ge 1}\ \sup_{\substack{s,t\in [0,1]\\ s\neq t}}\ \frac{\esp{\bigl|X^{n}(t)-X^{n}(s)\bigr|^{p}}}{|t-s|^{p/2}}<\infty.
  \end{equation*}
\end{theorem}
\begin{proof}
  If $|t-s|\le 2^{-n}$, there are at most two values of $k$ such that
  $r_k^n(t)-r_k^n(s)$ is not zero. Furthermore, we have
  \begin{equation*}
    |h_k^n(t)-h_k^n(s)| \le 2^{n/2}|t-s|
  \end{equation*}
  hence
  \begin{equation*}
    |h_k^n(t)-h_k^n(s)|^2\le 2^{n}|t-s|^{2}\le |t-s|.
  \end{equation*}
  According to Theorem~\ref{thm:mixing_wiener:maximal_inequality} with $m=2$,
  we have
  \begin{equation*}
    \esp{\Bigl|\sum_{i=0}^{2^{n}-1} X_i \,\Bigl(h_i^n(t)-h_i^n(s)\Bigr)\Bigr|
      ^{p}}\lesssim   |t-s|^{p/2}.
  \end{equation*}
  If $|t-s|>2^{-n}$, there are at most $\lfloor 2^{n}|t-s|\rfloor+2$ values of $k$ such that
  $r_k^n(t)-r_k^n(s)$ is not zero. Furthermore, for these values of~$k$, we have
  \begin{equation*}
    |h_k^n(t)-h_k^n(s)| \le 2^{-n/2}.
  \end{equation*}
  Apply Theorem~\ref{thm:mixing_wiener:maximal_inequality} with
  $m=\lfloor 2^{n}|t-s|\rfloor+2$ to get
  \begin{equation*}
    \esp{\Bigl|\sum_{i=0}^{2^{n}-1} X_i \,\Bigl(h_i^n(t)-h_i^n(s)\Bigr)\Bigr|
      ^{p}}\lesssim |t-s|^{p/2}.
  \end{equation*}
  The result follows.
\end{proof}
We are then in a position to reproduce the proof given in~\cite{coutin2020stein} of the
following theorem.
\begin{theorem}
  \label{thm:mixing_wiener:convergence}
  With the same hypotheses as above, for any $\eta<1/2$ and any $p>1/\eta$,
  we have
  \begin{equation*}
    \sup_{n\ge m\ge 1} 2^{m( 1/2-\eta )}\esp{\|X^{n}-\pi_m (X^{n})\|_{W_{\eta,p}}^p}^{1/p}<\infty.
  \end{equation*}
\end{theorem}
It is also well known (see~\cite{Friz2010}) that the same holds for Brownian motion.
\begin{theorem}
  \label{thm:eq_04-technical-lemmas-new:1}
  For any $\eta<1/2$ and any $p>1/\eta$,
  \begin{equation*}
    \sup_{m\ge 1} 2^{m( 1/2-\eta )}\esp{\|B-\pi_m (B)\|_{W_{\eta,p}}^p}^{1/p}<\infty.
  \end{equation*}
\end{theorem}
\subsection{Finite-dimensional convergence rate}
\label{sec:technical-lemmas}
\def\A{A}
\def\Astrong#1#2{{#1}^s(#2)}
\def\Aweak#1#2{{#1}^w(#2)}
\def\As#1#2#3{{#1}_{\vphantom{s}#2}^s(#3)}
\def\Aw#1#2#3{{#1}_{\vphantom{w}#2}^w(#3)}
\def\p{\partial}

Let $\A=\{0,\dots,2^{m}-1\}$. For $i\in \A$, we set
\begin{equation*}
  \A_{i}=\{i2^{n-m},\dots, (i+1)2^{n-m}-1\}.
\end{equation*}
We consider the $\ell^{1}$ norm on $\R^{2^{m}}$  defined as
\begin{equation*}
  \|(x_{0},\dots,x_{2^{m}-1})\|_{\ell^{1}}=\sum_{i\in \A}|x_{i}|.
\end{equation*}
We denote by $(e_{i},\, i\in \A)$ the canonical basis of $\R^{2^{m}}$.

\begin{theorem}
  \label{thm_mixing_core:reduction_finite_dimension}
  Let $X\in \mathcal{W}_{\delta}$.  For any $1\le m<n$, set
  \begin{equation*}
    \tau=2^{(m-n)/2}.
  \end{equation*}
  Then,   we have
  \begin{equation}
    \label{eq_clt_stable:1}\dist_{\KR( W_{\eta,p} )}\Big(\law( \pi_m(X^{n}) ),\ \law(\sigma \pi_m(B) )\Big)\lesssim 2^{-m(1/2-\eta)}\  \dist_{\KR( \R^{2^{m}} )}\Bigl(\law(\tau U_{m}^{n}),\ \law( B_{m})\Bigr),
  \end{equation}
  where for any $i\in \{0,\dots,2^{m}-1\}$,
  \begin{align*}
    U_{m,i}^{n} & =\sum_{j\in \A_{i}}X_{j}                                                                       \\
    B_{m,i}     & =\sigma 2^{m/2}\, \left[ B\left(\frac{i+1}{2^m}\right) - B\left(\frac{i}{2^m} \right) \right].
  \end{align*}
  and
  \begin{equation*}
    U_{m}^{n}=\sum_{i\in \A} U_{m,i}^{n}\ e_{i}\text{ and
    } B_{m}=\sum_{i\in \A}B_{m,i}\ e_{i}.
  \end{equation*}
\end{theorem}
\begin{proof}
  Using the equivalent characterization of the Wasserstein-1
  distance on $(\R^{2^{m}}, \ell^{1})$, we know that there exists a couple of
  random vectors $(\hat U_{m}^{n},\hat B_{m})$ built on the same probability
  space $(\hat{\Omega},\hat{\P})$ such that
  \begin{equation*}
    \dist_{\KR(\R^{2^{m}},\ell^{1})}\Bigl(\law(\tau U_{m}^{n}),\,\law( B_{m})\Bigr)=\hat{\mathbf E}\left[\left\|\tau\hat U_{m}^{n}-\hat B_{m}\right\|_{\ell^{1}}\right].
  \end{equation*}
  Since any $F\in \Lip_{1}(W_{\eta,p})$ is $1$-Lipschitz continuous, this entails that:
  \begin{multline*}
    \dist_{\KR( W_{\eta,p} )}\Big(\law( \pi_m(X^{n}) ), \, \law( \sigma\pi_m(B) )\Big) \\
    \begin{aligned}
       & = \sup_{F \in \Lip_{1}(W_{\eta,p})}\left(\esp{ F\Big( \pi_m(X^{n})\Big)} - \esp{F\Big( \sigma\pi_m(B)\Big)}\right)                                                                                                      \\
       & = \sup_{F \in \Lip_{1}(W_{\eta,p})}\hat{\mathbf E}\left[ F\left(\tau\sum_{q=0}^{2^{m}-1}\hat U_{m,q}^{n} h_{q}^{m}\right)\right]- \hat{\mathbf E}\left[ F\left(\sum_{q=0}^{2^{m}-1}\hat B_{m,q} h_{q}^{m}\right)\right] \\
       & \le  \dist_{\KR( \R^{2^{m}} )}\Bigl(\law( \tau U_{m}^{n} ),\ \law( B_{m} )\Bigr) \sup_{q\in\{0,\dots,2^{m}-1\}}\|h_{q}^{m}\|_{W_{\eta,p}}                                                                              \\
       & \lesssim 2^{-m(1/2-\eta)}  \dist_{\KR( \R^{2^{m}} )}\Bigl(\law( \tau U_{m}^{n} ),\ \law( B_{m} )\Bigr).                                                                                                                 
    \end{aligned}
  \end{multline*}
  The proof is thus complete.
\end{proof}
It is well known that we can represent the Kantorovitch--Rubinstein
distance  between any distribution~$\nu$ and a finite-dimensional Gaussian
distribution via the Stein--Dirichlet identity. We denote by $\nabla f$ (respectively $\Delta f$) the ordinary
gradient (respectively the Laplacian) of~$f$, a function
from $\R^{2^{m}}$ to $\R$, with enough regularity. The Ornstein--Uhlenbeck
semigroup associated with $\mu$, the centered Gaussian measure on $\R^{2^{m}}$
with covariance matrix $\sigma^{2}\operatorname{Id}$, is given by
(see \cite{Decreusefond2022})
\begin{equation}
  \label{eq_04-technical-lemmas-new:1}
  P_{t}f(x)=\int_{\R^{2^{m}}}f(e^{-t}x+\beta_{t}y)\dif \mu(y)
\end{equation}
where $\beta_{t}=\sqrt{1-e^{-2t}}$.  It is clear that
\begin{equation*}
  \lim_{t\to\infty}P_{t}f(x)=\int_{\R^{2^{m}}}f\dif \mu.
\end{equation*}
The infinitesimal generator of $(P_{t},\, t\ge 0)$ is given by
\begin{equation*}
  Lf(x)=\sp{x}{\nabla f(x)}_{\R^{2^{m}}}-\sigma^{2}\Delta f(x)
\end{equation*}
and satisfies
\begin{equation*}
  \frac{d}{dt}P_tf(x)=-LP_{t}f(x).
\end{equation*}
By the fundamental theorem of calculus, we have the following identity, which we
call the Stein--Dirichlet representation formula:
\begin{equation}
  \label{eq_mixing_core:4}\int_{\R^{2^{m}}}f\dif \nu-\int_{\R^{2^{m}}}f\dif \mu=\int_{\R^{2^{m}}}\int_{0}^{\infty}LP_{t}f(x)\dif t\dif \nu(x).
\end{equation}
Note that the function
\begin{equation*}
  \psi_{f}\, :\,  x\longmapsto -\int_{0}^{\infty}\left(P_{t}f(x)-\int_{\R^{2^{m}}}f\dif \mu\right)\dif t
\end{equation*}
is the solution of the so-called Stein equation $L\psi_{f}=f-\int f\dif \mu$.
Unfortunately, it is impossible to bound the third
derivative of $\psi_{f}$ (more precisely, the Lipschitz norm of its second-order
derivative) when $f$ is only $1$-Lipschitz. Indeed, since $f$ is at most
once differentiable, we have to resort to the Gaussian integration-by-parts
formula in order to differentiate $P_{t}f$: for any $k\ge 1$, for any $t>0$,
\begin{equation}
  \label{eq_mixing_core:5}\nabla^{(k)}(P_{t}f)(x)= e^{-t}\left( \frac{e^{-t}}{\beta_{t}}\right)^{k-1}\int \nabla f(e^{-t}x+\beta_{t}y)\otimes H_{k-1}(y)\dif \mu(y)
\end{equation}
where $H_{k}$ denotes the $k$-th tensorized Hermite polynomial, with $H_{0}=1$. The difficulty arises from the
fact that the function $t\mapsto \beta_{t}^{-k}$  is not integrable on $[0,1]$
for $k\ge 2$.
Thus, we rewrite  \eqref{eq_mixing_core:4} as
\begin{equation*}
  \int_{\R^{2^{m}}}f\dif \nu-\int_{\R^{2^{m}}}f\dif \mu= \int_{\R^{2^{m}}}\bigl(f(x)-P_\xi f(x)\bigr)\dif \nu(x)
  -\int_{\R^{2^{m}}}\int_{\xi}^{\infty}LP_{t}f(x)\dif t\dif \nu(x),
\end{equation*}
bound each term separately and then optimize with respect to~$\xi$.
Consequently,
for any other probability measure~$\nu$ on $\R^{2^{m}}$,
\begin{multline}
  \label{eq_mixing_core:2}\dist_{\KR}(\nu,\mu) \le \inf_{\xi \ge 0} \left[ \sup_{f\in \Lip_{1}(\R^{2^{m}})} \left|\int_{\R^{2^{m}}}(P_\xi f(x)-f(x))\dif \nu(x)\right|\right.\\ +
    \left.\sup_{f\in \Lip_{1}(\R^{2^{m}})} \left|\int_{\R^{2^{m}}}\int_\xi^{\infty} LP_{t}f(x)\dif t\dif \nu(x)\right|\right].
\end{multline}

For $\psi\, :\, \R^{2^{m}}\to \R$ and $k\ge 1$, $\nabla^{(k)}\psi$ belongs to
$(\R^{2^{m}})^{\otimes k}$, which we equip with the $\ell^{\infty}$ norm:
\begin{equation*}
  \|\nabla^{(k)}\psi\|_{\infty}=\sup_{x\in \R^{2^{m}}}\sup_{(i_{1},\dots,i_k)\in \A^{k}}\left|\p_{i_{1},\dots,i_{k}}^{(k)}\psi(x)\right|.
\end{equation*}
\begin{theorem}
  \label{thm:general_principle}
  Assume that the following
  two hypotheses hold:
  \begin{itemize}
    \item The function \(\mathfrak{n}\, :\, x\mapsto \|x\|_{\ell^{1}}\) is
          $\nu$-integrable. We set
          \begin{equation*}
            a=\int \mathfrak{n}(x)\dif \nu(x)+\int \|y\|_{\ell^{1}}\dif \mu(y).
          \end{equation*}
    \item There exists $b>0$, with $2b\le a$, such that for any
          $\psi$ regular enough, we have
          \begin{equation*}
            \int_{\R^{2^{m}}}|L\psi(x)|\dif \nu(x)\le b \sup\left(\|\nabla^{(i)} \psi\|_{\infty},\, i=1,2,3\right).
          \end{equation*}
  \end{itemize}
  Then,
  \begin{equation*}
    \dist_{\KR}(\nu,\mu)\lesssim  2b\left(1+\log\left(\frac{a}{2b}\right)\right).
  \end{equation*}
\end{theorem}
\begin{proof}
  According to~\eqref{eq_04-technical-lemmas-new:1}, we have
  \begin{equation*}
    \int_{\R^{2^m}}|P_\xi f(x)-f(x)|\dif \nu (x)\le \int_{\R^{2^m}}\int_{\R^{2^{m}}}|f(e^{-\xi}x+\sigma\beta_{\xi}y)-f(x)|\dif \mu(y)\dif \nu(x).
  \end{equation*}
  Since $f$ is Lipschitz continuous, we have
  \begin{align*}
    \int_{\R^{2^m}}|P_\xi f(x)-f(x)|\dif \nu (x) & \lesssim (1-e^{-\xi})\int_{\R^{2^m}}\|x\|_{\ell^{1}}\dif \nu (x)+\beta_{\xi}\int_{\R^{2^m}}\|y\|_{\ell^{1}}\dif \mu (y) \\
                                                 & \lesssim a\, \sqrt{1-e^{-\xi}}.
  \end{align*}
  We deduce from~\eqref{eq_mixing_core:5} that for $f\in \Lip_{1}(\R^{2^{m}})$,
  for any $\xi>0$, for any $(i,j,k)\in \A^{3}$, we have
  \begin{align*}
    \|\p_{i}\, P_{\xi}f\|_{\infty}    & \lesssim e^{-\xi}                               \\
    \|\p_{ij}^{2} P_{\xi}f\|_{\infty} & \lesssim \frac{e^{-\xi}}{\beta_{\xi}}           \\
    \|\p_{ijk}^{3}P_{\xi}f\|_{\infty} & \lesssim \frac{e^{-3\xi}}{\beta_{\xi}^{2}}\cdot
  \end{align*}
   Since $\beta_{\xi}^{-2}\le (1-e^{-\xi})^{-1}$, this means that
  \begin{equation}
    \label{eq_mixing_core:3}
    \sup_{(i,j,k)\in \A^{3}}\left( \|\p_{i}\, P_{\xi}f\|_{\infty},\, \|\p_{ij}^{2} P_{\xi}f\|_{\infty},\,  \|\p_{ijk}^{3}P_{\xi}f\|_{\infty}\right)\lesssim \frac{e^{-\xi}}{1-e^{-\xi}}\cdot
  \end{equation}
  Hence, by virtue of the second hypothesis, we have
  \begin{equation*}
    \left|\int_{\R^{2^{m}}}\int_\xi^{\infty} LP_{t}f(x)\dif t\dif \nu(x)\right|\lesssim b \int_{\xi}^{\infty}\frac{e^{-t}}{1-e^{-t}}\dif t.
  \end{equation*}
  Thus, we have to minimize
  \begin{equation*}
    a (1-e^{-\xi})^{1/2}+b \int_{\xi}^{\infty} \frac{e^{-t}}{1-e^{-t}}\dif t.
  \end{equation*}
  Since $\int_{\xi}^{\infty}(1-e^{-t})^{-1}e^{-t}\dif t=-\log(1-e^{-\xi})$, setting
  $u=\sqrt{1-e^{-\xi}}\in (0,1)$, this amounts to minimizing $u\mapsto au-2b\log u$
  over $(0,1)$. The minimum is attained at $u=2b/a$ and is equal to
  \begin{equation*}
    2b\left(1+\log\left(\frac{a}{2b}\right)\right).
  \end{equation*}
  The proof is thus complete.
\end{proof}
\begin{theorem}
  \label{thm:P_xi_minus_f}
  For any $\xi\ge 0$, for any $m\le n$,  we have
  \begin{equation}
    \label{eq_mixing_core:14}
    \sup_{f\in \Lip_{1}(\R^{2^{m}})} \esp{|P_{\xi}f(\tau U_{m}^{n})-f(\tau U_{m}^{n})|}\lesssim 2^{m}\sqrt{1-e^{-\xi}}.
  \end{equation}
  In other words, we can take $a=2^{m}$ in Theorem~\ref{thm:general_principle}.
\end{theorem}
\begin{proof}
  As a consequence of the maximal inequality (see
  Theorem~\ref{thm:mixing_wiener:maximal_inequality}), for any
  $f\in \Lip_{1}(\R^{2^{m}},\ell^{1})$, we get:
  \begin{equation*}
\esp{|P_{\xi}f(\tau U_{m}^{n})-f(\tau U_{m}^{n})|}\lesssim \sqrt{1-e^{-\xi}}\
    \esp{\|\tau U_{m}^{n}\|_{\ell^{1}}} + 2^{m}\beta_{\xi}
    \lesssim 2^{m}\sqrt{1-e^{-\xi}}.
\end{equation*}
  The proof is thus complete.
\end{proof}
Besides $m$, we introduce an additional parameter $k$, to be fixed later on but
assumed to satisfy
\begin{equation}
  \label{eq_mixing_core:13}k< n-m.
\end{equation}
For $M\subset \{0,\dots,2^{n}-1\}$ and $j\in \{0,\dots,2^{n}-1\}$, we define
the sets
\begin{equation*}
  \Aweak{M}{j}=\{q\in M,\, |q-j|>2^{k} \} \text{ and }
  \Astrong{M}{j}=\{q\in M,\, |q-j|\le 2^{k}\},
\end{equation*}
which correspond respectively to the indices of the random variables $X_{q}$
which are weakly (respectively strongly) dependent on $X_{j}$.
It may be useful to keep in mind Figure~\ref{fig:sets}, which shows how these
subsets are defined. We also need the following notation: for any
$M\subset \{0,\dots,2^{n}-1\}$,
\begin{align*}
  U_{M}          & =\sum_{i\in M}X_{i}                                         \\
  U              & =\sum_{i\in \A} U_{\A_{i}}\ e_{i}                           \\
  \Astrong{U}{j} & = \sum_{i\in \A} U_{\As{\A}{i}{j}}\ e_{i}                   \\
  \Aweak{U}{j}   & =\sum_{i\in \A} U_{\Aw{\A}{i}{j}}\ e_{i}                    \\
  \Aweak{U}{j,l} & =\sum_{i\in \A} U_{\Aw{\A}{i}{j}\cap \Aw{\A}{i}{l}}\ e_{i}.
\end{align*}
\begin{figure}[!ht]
  \centering
  \begin{tikzpicture}[font=\fontsize{6}{6}\selectfont,scale=0.4]
    \pgfmathsetmacro{\N}{8}       
    \pgfmathsetmacro{\NN}{16}
    \pgfmathsetmacro{\K}{1}       
    \pgfmathsetmacro{\i}{3}       
    \pgfmathsetmacro{\l}{5}       

    \draw[step=1, gray!20, very thin] (-\N/2,-\N/2) grid (1.5*\N,1.5*\N);
    \draw[very thin,color=blue!50] (-\N/2,0)  -- (1.5*\N,0);
    \draw[very thin,color=blue!50] (-\N/2,\N) -- (1.5*\N,\N);
    \draw[very thin,color=blue!50] (0,-\N/2) -- (0,1.5*\N);
    \draw[very thin,color=blue!50] (\N,-\N/2) -- (\N,1.5*\N);
    \draw[dashed, very thin, color=red!50,<->] (0,-\N/4)-- (\N,-\N/4) node[midway,below] {\(A_i\)};
    \draw[dashed, very thin, color=red!50,<-] (\N,-\N/4)-- (1.5*\N,-\N/4) node[midway,below] {\(A_{i+1}\)};
    \draw[dashed, very thin, color=red!50,->] (-\N/2,-\N/4)-- (0,-\N/4) node[midway,below] {\(A_{i-1}\)};
    \draw[dashed, very thin, color=red!50,<->] (-\N/4,0)-- (-\N/4,\N) node[midway,left] {\(A_i\)};
    \draw[dashed, very thin, color=red!50,<-] (-\N/4,\N)-- (-\N/4,1.5*\N    ) node[midway,left] {\(A_{i+1}\)};
    \draw[dashed, very thin, color=red!50,->] (-\N/4,-\N/2)-- (-\N/4,0) node[midway,left] {\(A_{i-1}\)};

    \foreach \x in {0,...,\N} {
        \foreach \y in {0,...,\N} {
            \filldraw[gray!50] (\x,\y) circle (2pt);
          }
      }

    \draw[very thin,red!50] (-\N/2,-\K-1) -- (1.5*\N-\K-1,1.5*\N);
    \draw[very thin,red!50] (-\K-1,-\N/2) -- (1.5*\N,1.5*\N-\K-1);
    \foreach \i in {0,...,\N} {
        \pgfmathsetmacro{\lower}{max(0,\i-2^\K)}
        \pgfmathsetmacro{\upper}{min(\N,\i+2^\K)}
        \foreach \j in {\lower,...,\upper} {
            \filldraw[red!60,opacity=0.5] (\i,\j) circle (4pt);
          }
      }

    \foreach \j in {0,...,\N} {
        \pgfmathparse{abs(\i-\j) > 2^\K ? 1 : 0}
        \ifnum\pgfmathresult>0
          \filldraw[blue] (\i,\j) circle (4pt);
        \fi
      }
    \foreach \j in {9,...,12} {
        \filldraw[blue!30] (\i,\j) circle (4pt);
      }
    \foreach \j in {-4,...,-1} {
        \filldraw[blue!50] (\i,\j) circle (4pt);
      }

    \foreach \j in {0,...,\N} {
        \pgfmathparse{abs(\i-\j) <= 2^\K ? 1 : 0}
        \ifnum\pgfmathresult>0
          \filldraw[red] (\i,\j) circle (4pt);
        \fi
      }

    \node at (-0.5,-0.5) {0};
    \node at (\N+0.5,-0.5) {\(2^{n-m}\)};
    \node at (-0.5,\N+0.5) {\(2^{n-m}\)};
    \node at (-0.5,\K+1.5) {\(2^k\)};
    \node at (\i,-0.5) {\(j\)};
    \node at (\l,-0.5) {\(l\)};
    \draw[thick,color=purple,opacity=0.5,<->] (\N,\N-\K-1) --node[midway,right,color=purple] {\(2.2^k\)} (\N,\N+\K+1) ;

      \begin{scope}[xshift=\N+10cm, yshift=0cm]
      \filldraw[red] (0,6) circle (4pt) node[right, xshift=0.5cm] {\(\As{\A}{i}{j}\)};
      \filldraw[blue] (0,5) circle (4pt) node[right, xshift=0.5cm] {\(\Aw{\A}{i}{j}\) };
      \filldraw[blue!30] (0,4) circle (4pt) node[right, xshift=0.5cm] {\(\Aw{\A}{i+1}{j}\) };
      \filldraw[blue!50] (0,3) circle (4pt) node[right, xshift=0.5cm] {\(\Aw{\A}{i-1}{j}\) };
    \end{scope}
  \end{tikzpicture}
  \caption{The different subsets of interest.}
  \label{fig:sets}
\end{figure}
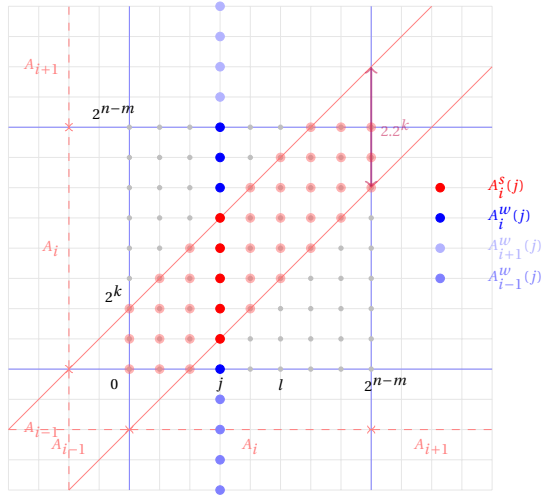

We make two remarks of paramount importance for the sequel. Since $k<n-m$, for
any $i\in \A$, if $j\in \A_{i}$ then
\begin{equation}
  \label{eq_mixing_core:1}\kappa\notin\{i-1,i,i+1\}\Longrightarrow \Aw{\A}{\kappa}{j}=\A_{\kappa} \text{
    and } \As{\A}{\kappa}{j}=\emptyset.
\end{equation}
Furthermore, we have
\begin{equation}
  \label{eq_mixing_core:7}\card\{j\in \A_{i},\ \As{\A}{\kappa}{j}\neq \emptyset \text{ for
  }\kappa\neq i\}\le 2^{k+1}.
\end{equation}
The main decomposition is given by the next theorem.
For $g\, :\, \R^{2^{m}}\to \R^{2^{m}}$, we denote by $(g_{i},\, i\in \A)$ its
components:
\begin{equation*}
  g=\sum_{i\in \A}g_{i}\, e_{i}.
\end{equation*}
\def\X{U}
\begin{theorem}
  \label{thm_mixing_wiener:decomposition}
  Let $X\in \mathcal{W}_{\delta}$ and $\tau=2^{(m-n)/2}$.
  For any $k<n-m$, for any sufficiently regular function
  $g\, :\, \R^{2^{m}}\to \R^{2^{m}}$, we have
  \begin{multline*}
    \esp{\sp{\tau U}{g(\tau U)}_{\R^{2^{m}}}}\\
    \begin{aligned}
       & =\tau^{2}\sum_{i\in \A} \sum_{j\in \A_i} \sum_{l\in \As{\A}{i}{j}}\Gamma_{|j-l|}\esp{ \p_i g_{i}(\tau U) }                                                                                                                 \\
       & +\tau \sum_{i\in \A} \sum_{j\in \A_i} \esp{X_{j}\, g_{i}(\tau \Aweak{U}{j} )}                                                                                                                                              \\
       & + \tau^{2}\sum_{i\in \A} \sum_{j\in \A_{i}} \sum_{l\in \As{\A}{i}{j}} \esp{\zeta_{j,l}\, \p_{i}g_{i}\left(\tau {\Aweak{U}{j,l}}\right)}                                                                                    \\
       & +\tau^{2 }    \sum_{i\in \A} \sum_{j\in \A_{i}} \sum_{l\in \As{\A}{i}{j}} \int_{0}^{1}    \esp{\zeta_{j,l}\,
                                                                                                     \left\{\p_{i}g_{i}\left(\tau \Aweak{U}{j}+\tau r \Astrong{U}{j}\right)-\p_{i}g_{i}\left(\tau \Aweak{U}{j,l}\right)\right\}
                                                                                                   }   \dif r      \\
       & +\tau^{2}      \sum_{i\in \A} \sum_{j\in \A_{i}} \sum_{l\in \As{\A}{i}{j}} \Gamma_{|l-j|}\int_{0}^{1}
      \esp{\p_{i}g_{i}\left(\tau \Aweak{U}{j}+r\tau\Astrong{U}{j}\right)-\p_{i}g_{i}(\tau U) }\dif r                                                                                                                                \\
       & +\tau^{2 }  \sum_{\varepsilon=\pm 1}  \sum_{i\in \A} \sum_{j\in \A_{i}}\sum_{l\in \As{\A}{i+\varepsilon}{j}}\int_{0}^{1}\esp{X_{l}X_{j}\p_{i+\varepsilon}g_{i}\left(\tau \Aweak{U}{j} +\tau r \Astrong{U}{j}\right)}\dif r \\
       & =\sum_{q=1}^{6}T_{q}.
    \end{aligned}
  \end{multline*}
\end{theorem}
\begin{proof}
  For $i\in \A$, let
  \begin{align*}
    R^{i} & =\tau\esp{U_{i}g_{i}\left(\tau U\right)}                                                              \\
          & =\tau \sum_{j\in \A_{i}} \esp{X_{j}g_{i}\left(\tau U\right) }                                         \\
          & =\tau\sum_{j\in \A_{i}}  \esp{X_{j}g_{i}\left(\tau \Aweak{U}{j}\right)}
    +\tau\sum_{j\in \A_{i}} \esp{X_{j}\left(g_{i}\left(\tau U\right)-g_{i}\left(\tau \Aweak{U}{j}\right) \right)} \\
          & =\tau R^{i}_{1}+\tau R^{i}_{2}.
  \end{align*}
  Apply a Taylor expansion to $R^{i}_{2}$,
  remembering~\eqref{eq_mixing_core:1}, to obtain
  \begin{align*}
    R^{i}_{2} & =\tau \sum_{j\in \A_{i}} \sum_{l\in \As{\A}{i}{j}}\int_{0}^{1}\esp{X_{j}X_{l}\p_{i}g_{i}\left(\tau \Aweak{U}{j}+r\tau\Astrong{U}{j}\right)}\dif r \\
              & +\tau\sum_{\varepsilon=\pm 1} \sum_{j\in \A_{i}}  \sum_{l\in \As{\A}{i+\varepsilon}{j}}
    \int_{0}^{1}\esp{X_{j}X_{l}\p_{i+\varepsilon}g_{i}\left(\tau \Aweak{U}{j}+r\tau\Astrong{U}{j}\right)}\dif r                                                   \\
              & =\tau R^{i}_{21}+\tau R^{i}_{22}.
  \end{align*}
  We again use the add-and-subtract technique
  to get
  \begin{align*}
    R^{i}_{21} & = \sum_{j\in \A_{i}}\sum_{l\in \As{\A}{i}{j}}\Gamma_{|j-l|}\esp{\p_{i} g_{i}(\tau U)}                                                                                                                   \\
               & +\sum_{j\in \A_{i}}\sum_{l\in \As{\A}{i}{j}}\Gamma_{|j-l|}\int_{0}^{1}\esp{\p_{i}g_{i}\left(\tau \Aweak{U}{j}+r\tau \Astrong{U}{j}\right)-\p_{i}g_{i}(\tau U)}\dif r                                    \\
               & +\sum_{j\in \A_{i}}\sum_{l\in \As{\A}{i}{j}}\esp{\zeta_{j,l}\p_{i}g_{i}(\tau\Aweak{U}{j,l})}                                                                                                            \\
               & +\sum_{j\in \A_{i}}\sum_{l\in \As{\A}{i}{j}}\int_{0}^{1}\esp{\zeta_{j,l}\left(\p_{i}g_{i}\left(\tau \Aweak{U}{j}+r\tau \Astrong{U}{j}\right)-\p_{i}g_{i}\left(\tau \Aweak{U}{j,l}\right)\right)}\dif r.
  \end{align*}
  Finally, we have
  \begin{equation*}
    \esp{\sp{\tau U}{g(\tau U)}_{\R^{2^{m}}}}=\sum_{i\in \A} R^{i}.
  \end{equation*}
  It remains to substitute the new expression for \( R^{i}_{21} \) into that of~\( R^{i}_{2} \), and then into that of~\( R^{i} \).
\end{proof}

\begin{theorem}[Bound for $T_{1}$]
  \label{lem_mixing_wiener:trace2}
  Let $X\in \mathcal{W}_{\delta,\phi}$. For any $i\in \A$, we have
  \begin{equation*}
    \left|\tau^{2}\sum_{j\in \A_i} \sum_{l\in \As{\A}{i}{j}}\Gamma_{|j-l|}-\sigma^{2}\right|\lesssim 2^{k-n+m}+\sum_{j> 2^{k}}\phi(j)^{\rho}
  \end{equation*}
  and consequently,
  \begin{equation}
    \label{eq_mixing_wiener:2}\left|T_{1}-\sigma^{2}\sum_{i\in \A}\esp{\p_{i}g_{i}(\tau U)}\right|\lesssim \|\nabla g\|_{\infty}\left(2^{k-n+2m}+2^{m}\sum_{j> 2^{k}}\phi(j)^{\rho}\right).
  \end{equation}
  With the hypothesis at hand, this yields
  \begin{equation}
    \label{eq_mixing_core:10}
    \left|T_{1}-\sigma^{2}\sum_{i\in \A}\esp{\p_{i}g_{i}(\tau U)}\right|\lesssim \|\nabla g\|_{\infty}\left(2^{k-n+2m}+2^{m+k(1-\rho\theta)}\right).
  \end{equation}
\end{theorem}
\begin{proof}
  Note that for any $i\in \A$,
  \begin{align*}
    \sum_{j\in \A_i} \sum_{l\in \As{\A}{i}{j}} \Gamma_{|j-l|}
     & =\sum_{j\in \A_0} \sum_{l\in \As{\A}{0}{j}} \Gamma_{|j-l|}                                   \\
     & =\sum_{z=0}^{2^{k}} \Gamma_{z}\ \card\left\{(j,l)\in \A_{0}\times \A_{0},\, |j-l|=z\right\}.
  \end{align*}
  It is clear that
  \begin{equation*}
    \card\left\{(j,l)\in \A_{0}\times \A_{0},\, |j-l|=z\right\}=
    \begin{cases}
      2^{n-m}      & \text{ if }z=0,           \\
      2(2^{n-m}-z) & \text{ if }0< z\le 2^{k}.
    \end{cases}
  \end{equation*}
  It follows that
  \begin{equation*}
    2(2^{n-m}-2^{k})\le \card\left\{(j,l)\in \A_{0}\times \A_{0},\, |j-l|=z\right\}\le 2\cdot 2^{n-m}.
  \end{equation*}
  Hence, we have
  \begin{align*}
    \left|\tau^{2} \sum_{j\in \A_0} \sum_{l\in \As{\A}{0}{j}}\Gamma_{|j-l|}-\sigma^{2}\right| & \lesssim 2^{k-n+m}\sum_{z=0}^{2^{k}}|\Gamma_{z}|+\sum_{z> 2^{k}}|\Gamma_{z}| \\
                                                                                              & \lesssim 2^{k-n+m} +\sum_{z>2^{k}}\phi(z)^{\rho}.
  \end{align*}
  We can now write
  \begin{align*}
    \left|T_{1}-\sigma^{2}\sum_{i\in \A}\esp{\p_{i}g_{i}(\tau U)}\right| & \lesssim \Bigl\|\sum_{i\in \A}\p_{i} g_{i}\Bigr\|_{\infty}\left|\tau^{2} \sum_{j\in \A_0} \sum_{l\in \As{\A}{0}{j}}\Gamma_{|j-l|}-\sigma^{2}\right| \\
                                                                         & \lesssim 2^{m}\|\nabla g\|_{\infty}\left|\tau^{2} \sum_{j\in \A_0} \sum_{l\in \As{\A}{0}{j}}\Gamma_{|j-l|}-\sigma^{2}\right|,
  \end{align*}
  and the result follows.
\end{proof}
\begin{theorem}[Bounds for $T_{2}$ and $T_{3}$]
  Let $X\in \mathcal{W}_{\delta,\phi}$.  With the same notations as above, we have:
  \begin{align}
    \label{eq_mixing_wiener:6}
    |T_{2}|                                   & \lesssim \|\nabla g\|_{\infty}\ \phi(2^{k})^{\frac{1+\delta}{2+\delta}}  \ 2^{\frac{3m}{2}+\frac{n}{2}} \\
    \label{eq_mixing_core:8}  |T_{3}| & \lesssim \|\nabla^{(2)}g\|_{\infty}\ \phi(2^{k})^{\frac{1+\delta}{2+\delta}}\  2^{k+2m}.
  \end{align}
\end{theorem}
\begin{proof}
  The two inequalities are proved in the same way.
  Apply Theorem~\ref{lem-hypo} with
  \begin{equation*}
    Y=X_j, \  Z=g_{i}(\tau\Aweak{U}{j})-g_{i}(0),\
    r=2+\delta,\   s=\frac{2+\delta}{1+\delta},\
    V_{1}=\{j\},\   V_{2}=\bigcup_{i'\in \A}\Aw{\A}{i'}{j},
  \end{equation*}
  to obtain
  \begin{align*}
    |T_2|
     & \lesssim  \|\nabla g\|_{\infty}\tau^{2}\phi(2^{k})^{\frac{1+\delta}{2+\delta}}\sum_{i\in \A}\sum_{j\in \A_{i}}\Bigl\|\,\|\Aweak{U}{j}\|_{\ell^{1}}\Bigr\|_{\L^{(2+\delta)/(1+\delta)}}.
  \end{align*}
  According to Theorem~\ref{thm:mixing_wiener:maximal_inequality},
  \begin{align*}
    \Bigl\|\,\|\Aweak{U}{j}\|_{\ell^{1}}\Bigr\|_{\L^{(2+\delta)/(1+\delta)}} & \le\sum_{i\in \A}\Bigl\| \sum_{l\in \Aw{\A}{i}{j}}X_{l}\Bigr\|_{\L^{(2+\delta)/(1+\delta)}} \\
                                                                             & \lesssim \sum_{i\in \A}  \sqrt{\card(\Aw{\A}{i}{j})}                                        \\
                                                                             & \lesssim 2^{(m+n)/2}.
  \end{align*}
  Assertion~\eqref{eq_mixing_wiener:6} then follows.
  We apply Theorem~\ref{lem-hypo} to
  \begin{multline*}
    Y=\zeta_{j,l},\  Z=\p_{i}g_{i}\left(\tau\Aweak{U}{j,l}\right) -\p_{i}g_{i}(0),\
    r=2+\delta, \\  s=\frac{2+\delta}{1+\delta},
    V_{1}=\{j,l\}, \  V_{2}=\bigcup_{i'\in \A}\left(\Aw{\A}{i'}{j}\cap \Aw{\A}{i'}{l}\right)
  \end{multline*}
  and according to Theorem~\ref{thm:mixing_wiener:maximal_inequality}, we get,
  with the same arguments as just above,
  \begin{align*}
    |T_3| & \lesssim
    \|\nabla^{(2)}g\|_{\infty}\tau^{3}\phi(2^{k})^{\frac{1+\delta}{2+\delta}}\sum_{i\in \A}\sum_{j\in \A_{i}}\sum_{l\in \As{\A}{i}{j}}\Bigl\|\,\|\Aweak{U}{j,l}\|_{\ell^{1}}\Bigr\|_{\L^{(2+\delta)/(1+\delta)}} \\
          & \lesssim  \|\nabla^{(2)}g\|_{\infty}\tau^{3}\phi(2^{k})^{\frac{1+\delta}{2+\delta}}2^{n}2^{k}2^{(m+n)/2}                                                                                             \\
          & \lesssim  \|\nabla^{(2)}g\|_{\infty} \phi(2^{k})^{\frac{1+\delta}{2+\delta}}2^{k+2m}.
  \end{align*}
  The proof is thus complete.
\end{proof}
The terms $T_{4}$, $T_{5}$ and $T_{6}$ are bounded above regardless of the rate of
convergence of~$\phi$ to zero.
\begin{theorem}[Bound for $T_{4}$]
  Let $X\in \mathcal{W}_{\delta,\phi}$. With the same notation as above, for
  $\delta\in (0,1]$, we have:
  \begin{equation}
    \label{eq_mixing_wiener:5}
    |T_{4}|\lesssim
    \|\nabla g\|_{\infty}^{1-\delta}\|\nabla^{(2)}g\|_{\infty}^\delta 2^{(1+\frac{\delta}{2})(m+k)-\frac{n\delta}{2}},
  \end{equation}
  for $\delta\ge 1$, we get
  \begin{equation}
    \label{eq_mixing_wiener:5_bis}
    |T_{4}|\lesssim
    \|\nabla^{(2)}g\|_{\infty}\ 2^{(1+\frac{1}{2})(m+k)-\frac{n}{2}}.
  \end{equation}
\end{theorem}
\begin{proof}
  We first remark that, for any $i\in \A$,
  \begin{equation*}
    \Aw{\A}{i}{j}\backslash \left(\Aw{\A}{i}{j}\cap\Aw{\A}{i}{l}\right)=\Aw{\A}{i}{j}\cap \As{\A}{i}{l}.
  \end{equation*}
  We then introduce the random vector
  \begin{equation*}
    U^{w,s}(j,l)=\sum_{i\in \A}U_{\Aw{\A}{i}{j}\cap \As{\A}{i}{l}}\, e_{i}.
  \end{equation*}
  We also recall that for any $\delta\ge 0$, we have
  $(|x|+|y|)^{\delta}\lesssim |x|^{\delta}+|y|^{\delta}$ for any $x,y \in \R$.
  We first write
  \begin{multline*}
    \left|  \p_{i}g_{i}\left(\tau \Aweak{U}{j}+\tau r \Astrong{U}{j}\right)-\p_{i}g_{i}\left(\tau \Aweak{U}{j,l}\right)\right|\\
    \begin{aligned}
       & \lesssim \|\nabla g\|_{\infty}^{1-\delta}\left|\p_{i}g_{i}\left(\tau \Aweak{U}{j}+\tau r \Astrong{U}{j}\right)-\p_{i}g_{i}\left(\tau \Aweak{U}{j,l}\right)\right|^{\delta}          \\
       & \lesssim \tau^{\delta}\|\nabla g\|_{\infty}^{1-\delta}\|\nabla^{(2)}g\|_{\infty}^{\delta}\left\{\|\Astrong{U}{j}\|_{\ell^{1}}^{\delta}+\|U^{w,s}(j,l)\|_{\ell^{1}}^{\delta}\right\}
    \end{aligned}
  \end{multline*}
  Then, we apply the $(1+2/\delta, 1+\delta/2)$-H\"older inequality to obtain
  \begin{multline*}
    |T_{4}|  \lesssim \|\nabla g\|_{\infty}^{1-\delta}\|\nabla^{(2)}g\|_{\infty}^{\delta}\ \tau^{2+\delta}\\
    \times \sum_{i\in \A} \sum_{j\in \A_{i}}
    \sum_{l\in \As{\A}{i}{j}} \|\zeta_{j,l}\|_{\L^{1+\delta/2}}\left\{\|\Astrong{U}{j}\|_{\L^{2+\delta}}^{\delta}+\|U^{w,s}(j,l)\|_{\L^{2+\delta}}^{\delta}\right\}.
  \end{multline*}
  We then note that, for $j,l$ fixed, the two vectors $\Astrong{U}{j}$ and $U^{w,s}(j,l)$
  have at most three non-zero components and that for any $i\in \A$,
  \begin{equation*}
    \card( \As{\A}{i}{j} )\le 2^{k} \text{ and
    }\card(\Aw{\A}{i}{j}\cap \As{\A}{i}{l})\le 2^{k},
  \end{equation*}
  hence according to
  Theorem~\ref{thm:mixing_wiener:maximal_inequality}, we have
  \begin{equation*}
    \|\Astrong{U}{j}\|_{\L^{2+\delta}}\lesssim 2^{k/2}
  \end{equation*}
  and the same result holds for $\|U^{w,s}(j,l)\|_{\L^{2+\delta}}$.
  The upper bound follows.

  For $\delta\ge 1$, the subadditivity of the function $x\mapsto x^{\delta}$ is replaced by a convexity inequality:
  \begin{equation*}
    (|x|+|y|)^{\delta}\le 2^{\delta-1}(|x|^{\delta}+|y|^{\delta})
  \end{equation*}
  which yields an insignificant multiplicative constant in the upper bound. The rest of the proof is similar.
\end{proof}
\begin{theorem}[Bound for $T_{5}$]
  Let $X\in \mathcal{W}_{\delta,\phi}$. With the same notations as above, we have:
  \begin{equation}
    \label{eq_mixing_wiener:3}|T_{5}|\lesssim 2^{\frac{k}{2}+\frac{3m}{2} -\frac{n}{2}}\ \|\nabla^{(2)}g\|_{\infty}.
  \end{equation}
\end{theorem}
\begin{proof}
  Note that we have:
  \begin{equation*}
    ( \tau \Aweak{U}{j}+r\tau \Astrong{U}{j} )-\tau U=\tau (r-1) \Astrong{U}{j}.
  \end{equation*}
  Since $\Gamma$ belongs to $\ell^{1/2}(\N)\subset \ell^1(\N)$, we have, by virtue of the maximal inequality,
  \begin{align*}
    |T_{5}| & \lesssim \tau^{3}\|\nabla^{(2)}g\|_{\infty} \sum_{i\in \A} \sum_{j\in \A_i} \sum_{l\in \As{\A}{i}{j}} |\Gamma_{|j-l|}|\ \sqrt{\card(\As{\A}{i}{j})} \\
            & \lesssim \|\nabla^{(2)}g\|_{\infty}  \tau^{3}2^{k/2}2^{n},
  \end{align*}
  and the result follows.
\end{proof}
\begin{theorem}[Bound for $T_{6}$]
  Let $X\in \mathcal{W}_{\delta,\phi}$. With the same notations as above, we have:
  \begin{equation}
    \label{eq_mixing_core:6}
    |T_{6}|\lesssim 2^{2k+2m-n}\ \|\nabla g\|_{\infty}.
  \end{equation}
\end{theorem}
\begin{proof}
  The Cauchy--Schwarz inequality and the weak stationarity of $X$ imply that, for
  $\varepsilon=\pm 1$,
  \begin{equation*}
    \left|  \esp{X_{l}X_{j}\p_{i+\varepsilon}g_{i}\left(\tau \Aweak{U}{j} +\tau r \Astrong{U}{j}\right)}\right|\le \|\nabla g\|_{\infty} \esp{X_{0}^{2}}.
  \end{equation*}
  We now need to estimate the cardinality of
  \begin{equation*}
    \bigl\{(i,j,l), i\in \A, j\in \A_{i},\, l\in \As{\A}{i-1}{j}\cup \As{\A}{i+1}{j}\bigr\}.
  \end{equation*}
  Since
  \begin{equation*}
    \card\{j\in \A_{i},\ \As{\A}{i-1}{j}\cup \As{\A}{i+1}{j}\neq \emptyset \}\le 2^{k+1},
  \end{equation*}
  and
  \begin{equation*}
    \card\left(\As{\A}{i-1}{j}\cup \As{\A}{i+1}{j}\right)\le 2^{k+1},
  \end{equation*}
  we obtain the desired result.
\end{proof}
\begin{theorem}[Optimization]
  \label{thm:optimisation}
  Let $X\in \mathcal{W}_{\delta,\phi}^{\theta}$, let $\eta<1/2$ and let $p>1/\eta$.
  Set
  \begin{equation}
    \label{eq_mixing_core:15}
    \bar{\delta}=\min(\delta,\,1)\ \text{ and }\ c=\frac{\bar{\delta}}{2+\bar{\delta}}\in \Bigl(0,\frac{1}{3}\Bigr],
  \end{equation}
  and assume that
  \begin{equation}
    \label{eq_mixing_core:16}
    2\,\rho\theta\, c>1\Longleftrightarrow \rho\theta>\frac{2+\bar{\delta}}{2\,\bar{\delta}}\cdot
  \end{equation}
  Set
  \begin{equation}
    \label{eq_mixing_core:12}
    \gamma^{*}=\frac{2\,\rho\theta\, c-1}{2\,\rho\theta+3}
    \in \Bigl(0,\frac{1}{3}\Bigr)\cdot
  \end{equation}
  Then, we have
  \begin{equation*}
    \dist_{\KR(W_{\eta,p})}\Bigl(\law(\pi_{m}(X^{n})),\ \law(\sigma\pi_{m}(B))\Bigr)\lesssim N^{-\gamma^{*}(1/2-\eta)}\log N
  \end{equation*}
  and
  \begin{equation*}
    \dist_{\KR(W_{\eta,p})}\left(\law\Bigl(\sum_{i=0}^{2^{n}-1}X_{i}h_{i}^{n}\Bigr),\ \law(\sigma B)\right)\lesssim N^{-\gamma^{*}( 1/2-\eta )}\log N,
  \end{equation*}
  where $N=2^{n}$.
  Since $\gamma^{*}$ depends on $\delta$ only through $\bar{\delta}=\min(\delta,1)$,
  the rate obtained for any $\delta\ge 1$ is the same as the one obtained for
  $\delta=1$, namely
  \begin{equation*}
    \gamma^{*} = \frac{2 \, \rho\theta - 3}{3 \, {\left(2 \, \rho\theta+3\right)}}\cdot
  \end{equation*}
\end{theorem}
\begin{proof}
  For $F \in \Lip_{1}(W_{\eta,p})$, we have
  \begin{align*}
    \esp{F(X^{n})}-\esp{F(\sigma B)} & \le\left|\esp{F(X^{n})}-\esp{F(\pi_{m}(X^{n}))}\right|             \\
                                     & + \left| \esp{F(\pi_{m}(X^{n}))}-\esp{F(\sigma\pi_{m}( B))}\right| \\
                                     & + \left|\esp{F(\sigma\pi_{m}( B))}-\esp{F(\sigma B)}\right|.
  \end{align*}
  According to Theorems~\ref{thm:mixing_wiener:convergence}
  and~\ref{thm:eq_04-technical-lemmas-new:1}, the two extreme terms of the
  right-hand side are bounded, up to a multiplicative constant, by
  $2^{-m(1/2-\eta)}$. As to the middle term,
  Theorem~\ref{thm_mixing_core:reduction_finite_dimension} bounds it by
  $2^{-m(1/2-\eta)}\dist_{\KR(\R^{2^{m}})}(\law(\tau U_{m}^{n}),\, \law(B_{m}))$,
  and Theorem~\ref{thm:general_principle} applies to this last distance with
  \begin{equation*}
    a=a_{n,m}=\esp{\|\tau U_{m}^{n}\|_{\ell^{1}}}\lesssim 2^{m},
  \end{equation*}
  see Theorem~\ref{thm:P_xi_minus_f}, and with $b=b_{n,m,k}$ given by the
  estimates of $T_{1},\dots,T_{6}$, namely
  \begin{multline}
    \label{eq_mixing_core:19}
    b_{n,m,k}  =  2^{k+2m-n}+2^{k(1-\rho\theta)+m}+ 2^{-\rho\theta k+3m/2+n/2}\\
    + 2^{(1-\rho\theta) k+2m}+2^{(1+\bar{\delta}/2) k+(1+\bar{\delta}/2)m-n\bar{\delta}/2}+2^{k/2+3m/2-n/2}+2^{2k+2m-n}.
  \end{multline}
  We thus have
  \begin{equation}
    \label{eq_mixing_core:23}
    \left| \esp{F(\pi_{m}(X^{n}))}-\esp{F(\sigma\pi_{m}( B))}\right|\lesssim 2^{-m(1/2-\eta)}\ b_{n,m,k}\left( 1+\log\frac{a_{n,m}}{2\,b_{n,m,k}} \right).
  \end{equation}
  The fifth term comes from the bound obtained for $T_{4}$: the interpolation
  argument used there requires $\delta\le 1$, and for $\delta\ge 1$ we can only
  use the same estimate with $\delta$ replaced by~$1$,
  see~\eqref{eq_mixing_wiener:5} and~\eqref{eq_mixing_wiener:5_bis}. This is why
  $\bar{\delta}=\delta\wedge 1$, and not $\delta$, appears
  in~\eqref{eq_mixing_core:19}: increasing the number of moments beyond
  $2+\bar\delta$ does not improve the estimate of $T_{4}$, hence does not improve
  the final rate.

  Since the two extreme terms are of order $2^{-m(1/2-\eta)}$, and since the
  logarithmic factor of~\eqref{eq_mixing_core:23} is at most of
  order~$n$, the best rate we can hope for is $2^{-m(1/2-\eta)}\log n$, and it is
  achieved as soon as $b_{n,m,k}$ remains bounded. We are thus led to maximize
  $m$ under this last constraint.

  To this end, we parameterize the triple $(n,m,k)$ by $m=\gamma n $ with
  $0<\gamma<1$ and $k=\alpha n$ with $0<\alpha<1$. Note that the
  hypothesis~\eqref{eq_mixing_core:13} implies $\alpha<1-\gamma$. Each of
  the seven terms of $b_{n,m,k}$ is of the form $2^{ne}$, where $e$ is an affine
  function of $(\alpha,\gamma)$, so that $b_{n,m,k}$ is bounded as soon as these
  seven exponents are non-positive. Listed in the same order as the
  corresponding terms, these conditions read
  \begin{align*}
    \alpha & \le 1-2\gamma                       \\
    \alpha & \ge \frac{\gamma}{\rho\theta-1}     \\
    \alpha & \ge\frac{1}{2\rho\theta}(1+3\gamma) \\
    \alpha & \ge \frac{2}{\rho\theta-1}\gamma    \\
    \alpha & \le c-\gamma                        \\
    \alpha & \le {1}-3\gamma                     \\
    \alpha & \le\frac{1}{2}-\gamma.
  \end{align*}
  Among the three conditions $\alpha<1-\gamma$, $\alpha\le 1-2\gamma$ and
  $\alpha\le 1-3\gamma$, the last one is the strongest, hence we may discard the
  other two. Likewise, the fourth inequality is stronger than the second one,
  hence the latter can be discarded. Moreover, since the constraint $\alpha>0$
  forces $\gamma<c\le 1/3$, we have $2\gamma<1-c$ and therefore
  \begin{equation*}
    c-\gamma\le \min\left( \frac{1}{2}-\gamma,\ 1-3\gamma \right),
  \end{equation*}
  so that the fifth inequality is the strongest of the three upper bounds. Note
  that the lines $\alpha=c-\gamma$, $\alpha=1/2-\gamma$ and $\alpha=1-\gamma$ are
  parallel, hence which of them is the lowest does not depend on $\gamma$; since
  $c\le 1/3<1/2<1$, it is always the first one. We are thus left with the two
  constraints
  \begin{equation*}
    0<\alpha \le f_{1}(\gamma):= c-\gamma
  \end{equation*}
  and
  \begin{equation*}
    \alpha\ge \max\bigl(f_{2}(\gamma),\, f_{3}(\gamma)\bigr), \text{ where }
    f_{2}(\gamma):=\frac{2}{\rho\theta-1}\gamma \text{ and } f_{3}(\gamma):=\frac{1}{2\rho\theta}(1+3\gamma).
  \end{equation*}
  The graphs of $f_{2}$ and $f_{3}$ intersect at the point
  \begin{equation*}
    \left(\frac{\rho\theta-1}{\rho\theta+3},\, \frac{2}{\rho\theta+3}\right),
  \end{equation*}
  which lies strictly above the graph of $f_{1}$: this amounts to
  $\rho\theta+1>c\,(\rho\theta+3)$, which holds
  since $c\le 1/3$.
  It follows that $\gamma^{*}$ is the abscissa of the intersection point of the graphs
  of $f_{3}$ and $f_{1}$, which gives~\eqref{eq_mixing_core:12}. Since $\gamma^*>0$, we must have $2\rho \theta c>1$, which is condition~\eqref{eq_mixing_core:16}.
\end{proof}
Note that $\gamma=1/3$ would yield the exponent $(-1/6+\eta/3)$, which
is the same as the one we obtained in~\cite{coutin2020stein} for independent
random variables. The restriction obtained here shows that this threshold
cannot be reached in the presence of correlations.

For $\delta\ge 1$ and $\theta$ arbitrarily large, which corresponds to
super-polynomially fast mixing (as in~\cite{Haeusler1984}) or even to
independence, $\gamma^{*}$ can be chosen as close to $1/3$ as we
desire. This is the result of~\cite{coutin2020stein}.
\def\DD{\mathbb D}
We denote by $\DD$ the Skorohod space of right-continuous functions on $[0,1]$
with left-hand limits, equipped with the uniform norm~$\L^{\infty}$. Consider now
the piecewise constant version of the random walk:
\begin{equation*}
  Y^{n}(t)=2^{-n/2}\sum_{k=0}^{\lfloor 2^{n}t\rfloor-1}X_{k}.
\end{equation*}
We can proceed similarly for the Kantorovitch--Rubinstein distance on the space
$(\DD,\, \L^{\infty})$, the only two changes being the following. First,
the affine interpolation $\pi_{m}X^{n}$ still approximates $Y^{n}$, but the rate
is now governed by the integrability exponent~$p$ instead of the regularity
exponent~$\eta$; this is the content of Lemma~\ref{lem:skorohod_interpolation}
below. Second, in the proof of
Theorem~\ref{thm_mixing_core:reduction_finite_dimension}, the quantity
$\sup_{q}\|h_{q}^{m}\|_{W_{\eta,p}}$ has to be replaced by
\begin{equation}
  \label{eq_mixing_core:22}
  \sup_{0\le q<2^{m}}\|h_{q}^{m}\|_{\L^{\infty}}=2^{-m/2},
\end{equation}
since $h_{q}^{m}$ increases from~$0$ to $2^{-m/2}$ on
$[q2^{-m},(q+1)2^{-m}]$ and is constant afterwards.
\begin{lemma}
  \label{lem:skorohod_interpolation}
  Let $X\in \mathcal{W}_{\delta,\phi}$ and let $p\in [2,2+\delta]$. Then
  \begin{equation}
    \label{eq_mixing_core:20}
    \sup_{n> m\ge 1}\ 2^{m(1/2-1/p)}\ \esp{\|Y^{n}-\pi_{m} X^{n}\|_{\L^{\infty}}^{p}}^{1/p}<\infty,
  \end{equation}
  and the same holds with $Y^{n}-\pi_{m}X^{n}$ replaced by $B-\pi_{m}B$.
\end{lemma}
\begin{proof}
  Let $S_{k}(i)=\sum_{j=k}^{k+i-1}X_{j}$ as in
  Theorem~\ref{thm:mixing_wiener:maximal_inequality}. Fix
  $q\in \{0,\dots,2^{m}-1\}$ and $t\in [q2^{-m},\, (q+1)2^{-m})$. Since
  $\pi_{m}X^{n}$ is affine on this interval, $\pi_{m}X^{n}(t)$ lies between
  $X^{n}(q2^{-m})$ and $X^{n}((q+1)2^{-m})$, and these two values differ by
  $2^{-n/2}|S_{q2^{n-m}}(2^{n-m})|$. Furthermore, $Y^{n}(t)$ and $X^{n}(q2^{-m})$
  differ by at most $2^{-n/2}\max_{1\le i\le 2^{n-m}+1}|S_{q2^{n-m}}(i)|$. Hence
  \begin{equation*}
    \|Y^{n}-\pi_{m}X^{n}\|_{\L^{\infty}}\le 2\cdot 2^{-n/2}\ \max_{0\le q<2^{m}}\ \max_{1\le i\le 2^{n-m}+1}\bigl|S_{q2^{n-m}}(i)\bigr|.
  \end{equation*}
  Bounding the maximum over~$q$ by the sum over~$q$ and applying
  Theorem~\ref{thm:mixing_wiener:maximal_inequality} to each block, we get
  \begin{equation*}
    \esp{\|Y^{n}-\pi_{m}X^{n}\|_{\L^{\infty}}^{p}}\lesssim 2^{-np/2}\ \sum_{q=0}^{2^{m}-1}\esp{\max_{1\le i\le 2^{n-m}+1}\bigl|S_{q2^{n-m}}(i)\bigr|^{p}}\lesssim 2^{-np/2}\ 2^{m}\ 2^{(n-m)p/2},
  \end{equation*}
  that is to say $\esp{\|Y^{n}-\pi_{m}X^{n}\|_{\L^{\infty}}^{p}}^{1/p}\lesssim
    2^{-m(1/2-1/p)}$. The very same computation applies to $B-\pi_{m}B$, using
  $\esp{\sup_{s\le h}|B_{s}|^{p}}^{1/p}\lesssim h^{1/2}$ instead of the maximal
  inequality.
\end{proof}
By the very same techniques as before, we arrive at the following rate of convergence.
\begin{theorem}
  \label{thm:optimisation_skorohod}
  Let $X\in \mathcal{W}_{\delta,\phi}^{\theta}$ and let $p\in (2,\,2+\delta]$. Let
  $\bar{\delta}$ and $c$ be as in~\eqref{eq_mixing_core:15} and assume
  that~\eqref{eq_mixing_core:16} holds. Set
  \begin{equation*}
    \label{eq_mixing_core:21}
    \tilde\gamma^{*}=\frac{2\rho\theta\bar{\delta}-\bar{\delta}-2}{(2+\bar{\delta})(2\rho\theta+3)-\dfrac{2}{p}\,(2\rho\theta+2+\bar{\delta})}\cdot
  \end{equation*}
  Then, we have
  \begin{equation*}
    \label{eq_mixing_core:9b}
    \dist_{\KR(\DD,\L^{\infty})}\left(\law(Y^{n}),\ \law(\sigma B)\right)\lesssim N^{-\tilde\gamma^{*}( 1/2-1/p )}\log N
  \end{equation*}
  where $N=2^{n}$.

  Moreover, $\tilde\gamma^{*}>\gamma^{*}$ and
  $\tilde\gamma^{*}\to \gamma^{*}$ as $p$ goes to infinity. The exponent
  $\tilde\gamma^{*}(1/2-1/p)$ is increasing with respect to~$p$: the best rate is
  thus obtained for $p=2+\delta$.
\end{theorem}

\section*{Declarations}

\begin{itemize}
  \item Funding:
        LD was supported by the French National Research Agency (ANR)
        \textit{via} the project no.~ANR-22-PEFT-0010 of the France
        2030 programme PEPR R\'eseaux du Futur.

  \item AI: Claude Opus 5 was used to verify the computations of the
        optimization problems in the proofs of
        Theorems~\ref{thm:optimisation} and~\ref{thm:optimisation_skorohod}.
\end{itemize}

\bibliography{BibFile} \bibliographystyle{amsplain}
\end{document}